\documentclass{birkjour}
 \newtheorem{theorem}{Theorem}[section]
 \newtheorem{corollary}[theorem]{Corollary}
 
 \newtheorem{proposition}[theorem]{Proposition}
 \theoremstyle{definition}
 \newtheorem{definition}[theorem]{Definition}
 \theoremstyle{remark}
 
 \newtheorem*{example}{Example}
 \numberwithin{equation}{section}

\begin{document}

%
%
%
%
%
%
%
%
%

\title[General factorial ]
 {Generalized factorials characterized by\br Dirichlet convolution}

\author[Wanli Ma]{Wanli Ma}

\address{
School of Mathematical Sciences\br
East China Normal University\br 
500 Dongchuan Road (200241)\br
Shanghai, PR China
}
\email{mawanli271828314159@gmail.com}

\thanks{This work was completed with the support of Prof. Zhiguo Liu}
\subjclass{Primary 11B65; Secondary 11Z05.}

\keywords{factorial, arithmetic function, Dirichlet convolution, superadditive sequence.}

\date{December 24, 2024}

\begin{abstract}
We extend A.B. Mingarelli's method for constructing generalized factorials.  Our extension uses a pair of arithmetic functions $(x, y)$, where $x$ is superadditive.  When $x$ is the identity function, our generalized factorial reduces to Mingarelli's. A result on the irrationality of the Euler constant within this framework is given. Using Dirichlet convolution, we characterize when two pairs $(\alpha, \beta)$ and $(x, y)$ generate the same factorials.  
\end{abstract}

\maketitle
\section{Introduction}
	
	The study of generalized factorials has attracted considerable interest over the past century. Introductions and reviews of this area are given in \cite{Bhargava1998general,Bhargava2000factorial,chabert2006oldproblems,knuth1989power}. Generalized factorials are usually defined in association with subsets of integers $E \subset \mathbb{Z}$ (in this paper, $E$ is a pair of arithmetic functions) \cite{Bhargava2000factorial,mingarelli2013abstract}. The $n$-th factorial associated with $E$ is usually denoted by $n!_E$.

	The following definition is adapted from Mingarelli \cite{mingarelli2013abstract}, who referred to these as "abstract factorials."
	\begin{definition}
		A \textbf{generalized factorial} is a function $!_E: \mathbb{N}\to \mathbb{Z}^+$ that satisfies the following three conditions:
		
			 $0!_E=1.$
			 
			 For all non-negative integers $n$ and $k$, where $0\leq k\leq n$, the generalized binomial coefficient is a positive integer:
			\[\binom{n}{k}_E:=\frac{n!_E}{k!_E(n-k)!_E}\in \mathbb{Z}^+.
			\]
			
			 For every positive integer $n$, $n!$ divides $n!_E$.
		
	\end{definition}
	Generalized factorials have several characterizations \cite{mingarelli2013abstract}, and the following one is easily verifiable:
	\begin{proposition}
		Let $!_E$ be a generalized factorial. Then there exists a sequence of positive integers $(h_n)_{n\geq 0}$ with $h_0=1$ such that for each $n\in \mathbb{N}$,
		\begin{equation}\label{factor of abstract factorial}
			\left. h_k h_{n-k} \middle| h_n\binom{n}{k}, \; k=0,1,2,\ldots,n.\right.
		\end{equation}
		Conversely, if there exists a sequence of positive integers $(h_n)_{n\geq 0}$ satisfying \eqref{factor of abstract factorial} and $h_0=1$, then the function $!_E:\mathbb{N}\to \mathbb{Z}$ defined by $n!_E:=n!h_n$ is a generalized factorial.  
	\end{proposition}
	\begin{corollary}
		Let $(h_n)_{n\geq 0}$ be a sequence of integers such that $h_0=1$ and $h_kh_{n-k} \mid h_n$ for $0\leq k\leq n$ and each $n\in \mathbb{N}$. Then the sequence defined by $n!_E:=n!h_n$ is a generalized factorial.
	\end{corollary}
	
	The Dirichlet convolution of two arithmetic functions $\alpha$ and $\beta$, denoted by $\alpha\star\beta$, is defined as:
	\[(\alpha \star \beta)(n):= \sum_{d|n} \alpha\left(\frac{n}{d}\right)\beta(d), \quad n\geq 1.
	\]
	Dirichlet convolution is associative and commutative. The convolution identity, denoted by $\epsilon$, is defined as:
	\[ \epsilon(n)=\delta_{1,n}, \quad n\geq 1.
	\]
	where $\delta_{1,n}$ is the Kronecker delta. For any arithmetic function $\alpha$, $\epsilon$ satisfies the formula $\alpha \star \epsilon=\epsilon \star \alpha=\alpha$.
	The Dirichlet inverse of an arithmetic function $\alpha$ is a function $\beta$ such that $\alpha \star \beta=\beta\star\alpha=\epsilon$. For the constant function $1(n)=1$, its Dirichlet inverse is the Möbius function $\mu$.
	The Von Mangoldt function, denoted by $\Lambda$, is defined as:
	\[\Lambda(n)=\begin{cases}
		\log(p) & \text{if } n=p^k \text{ for some prime }p \text{ and integer } k\geq 1,\\
		0 & \text{otherwise}.
	\end{cases}
	\]
	
	In the next section, we shall consider those generalized factorials induced by sequences $(h_n)$ such that $h_kh_{n-k} \mid h_n$. In the final section, we will use Dirichlet convolution to characterize these generalized factorials.
	
\section{Construction and Examples}

	In the following, $x=(x(n))_{n\geq1}$ and $y=(y(n))_{n\geq 1}$ will always represent sequences of positive integers, or equivalently, arithmetic functions. We shall use the pair $(x,y)$ to construct a class of generalized factorials.
\subsection{Construction of Generalized Factorials }
	
	Recall that a sequence of positive integers (or arithmetic function) $x=(x(n))_{n\geq1}$ is called \textit{superadditive} if
	\[x(m)+x(n)\leq x(m+n)\quad \text{for all } m,n \geq 1.
	\]
	
	For a general non-decreasing sequence $a=(a(n))_{n\geq 1}$ that is unbounded above, we can define a subsequence of $a$, denoted $\hat{a}=(\hat{a}(n))_{n\geq 1}$, that is superadditive and minimal. This means that for any other superadditive subsequence $b=(b(n))_{n\geq 1}$ of $a$, we have $\hat{a}$ bounded by $b$ term wise:
	\[\hat{a}(n)\leq b(n), \quad \text{for all } n\geq 1.
	\]
	
	The minimal superadditive subsequence of an unbounded sequence can be defined by a greedy algorithm, described as follows:
	\begin{definition}\label{defi of minsuper sequence}
			For a non-decreasing sequence $x=(x(n))_{n\geq 1}$ of non-negative integers that is unbounded, define $\hat{x}=(\hat{x}(n))_{n\geq 1}$ recursively as:
			
			$\hat{x}(1):=x(1)$, which is the minimal element in $x$.
			
			For $n>1$, if $\hat{x}(1),\hat{x}(2),\ldots, \hat{x}(n-1)$ have been defined, define
			\[\hat{x}(n):=\min_{a\in \text{Im}(x)}\{ a \mid \hat{x}(k)+\hat{x}(n-k)\leq a, \text{ for } k=1,2,\ldots,n-1\}.
			\]
	\end{definition} 
	Here, to facilitate our definition, we identify the function $x$ with the image of $x$. 
	It is straightforward to verify that the subsequence $\hat{x}$ defined above is indeed superadditive and minimal.
	\begin{example}
		For each positive integer $k$, function $x$ defined by $x(n)=n^k$ is superadditive itself. Hence, we have $\hat{x}=x$. This is because for all positive integers $m$ and $n$, $m^k+n^k\leq (m+n)^k.$
	\end{example}
	%
	%
	\begin{definition}\label{definition of array alpha(n,k)}
		Let $x=(x(n))_{n\geq 1}$ be an unbounded sequence of positive integers. Define an array of elements $x(n,k)_{n\geq 1, k\geq 1}$ recursively by the following minimal condition:
		\begin{align*}
			& x(1,1)=\min{x},\\
			& x(n,k)=0 \text{ whenever } n<k,\\
			& x(n,k)=\min_{a \in Im(x)} \{a \mid x(i,k)+x(n-i,k)\leq a, \text{ for } 1\leq i \leq n-1\}.
		\end{align*}
	\end{definition}
	\begin{example}
		Let $x=(1,2,3,4,5,\ldots)$ be the usual sequence of positive integers. Then:
		$x(1,1)=1$, since 1 is the minimal element in $x$.
		$x(2,1)$ is the minimal element in $x$ such that $x(1,1)+x(1,1)\leq x(2,1)$, hence $x(2,1)=2$.
		Since $x(1,2)=0$, we take $x(2,2)$ as the minimal element in $x$.
		Continuing this way, we have $x(3,1)=3$, $x(3,2)=1$, $x(3,3)=1$, and so on. 
	\end{example}
	Now we have all the necessary ingredients to define the abstract factorial based on two sequences $x$ and $y$:
	\begin{definition}\label{defi of factorial set Bn}
		Let $x=(x(n))_{n\geq 1}$ and $y=(y(n))_{n\geq 1}$ be two sequences of positive integers, where $x$ is unbounded above. Let $x(n,k)$ be the array associated with $x$ as defined in Definition \ref{definition of array alpha(n,k)}. 
		We define the \textbf{factorial set associated with $(x,y)$}, denoted by $B=(B_n)_{n\geq 0}$, as follows:
		
		$B_0:=1$,
		
		For $n\geq 1$,
		\[B_n:=\prod_{k=1}^{n}y(k)^{x(n,k)}=y(1)^{x(n,1)}y(2)^{x(n,2)}\cdots y(n)^{x(n,n)}.
		\]
	\end{definition}
	It is worth noting that by our construction, for any $m\in \mathbb{N}$, $B_n$ divides $B_{n+m}$. To see this:
	\[\frac{B_{n+m}}{B_n}=\prod_{i=1}^{n}y(i)^{x(n+m,i)-x(n,i)}\prod_{j=n+1}^{n+m}y(j)^{x(n+m,j)}.
	\]
	By the definition of $x(n,k)$, we have $x(n+m,i)-x(n,i)\geq 0$ for all $1\leq i\leq n$. 
	\begin{theorem}\label{the general definition of abstract factorial}
		Let $x=(x_n)_{n\geq 1}$ be an unbounded sequence of positive integers, and $y=(y_n)_{n\geq 1}$ be a sequence of positive integers. Let $(B_n)_{n\geq 0}$ be the factorial set associated with $(x,y)$ as defined in Definition \ref{defi of factorial set Bn}.
		Then, for each $m\geq 0$, the map $!_{x,y}:\mathbb{N}\to \mathbb{Z}^+$ defined by:
		\begin{align*}
			0!_{x,y} &:= 1,\\
			n!_{x,y} &:= n!B_1B_2\cdots B_{n+m}.
		\end{align*}
		is an abstract factorial.
	\end{theorem}
	\begin{proof}
		The only part that needs some explanation is the integer nature of the generalized binomial coefficients. For a fixed $m\in \mathbb{N}$, we have:
		
		\[\binom{n}{r}_{x,y}:=\frac{n!_{x,y}}{r!_{x,y}(n-r)!_{x,y}}=\binom{n}{r}\prod_{ i=1}^{n-r+m}\frac{B_{r+m+i}}{B_{i}}.
		\]
		
		Here, $1\leq r \leq n-1$, and the other cases are trivial.
		Since we have shown earlier that $B_i$ divides $B_{r+m+i}$ for all $i$. This completes the proof.
	\end{proof}
	\begin{theorem}\label{the simpler definition of abstract factorial}
		Let $(B_n)_{n\geq 0}$ be the factorial set associated with the sequences $x$ and $y$ as defined in Definition \ref{defi of factorial set Bn}. If $n!$ divides $B_n$ for all $n\geq 0$, then $(B_n)_{n\geq 0}$ is a generalized factorial.
	\end{theorem}
	\begin{proof}
		It will suffice to show that for $0\leq k \leq n$, $B_n/(B_kB_{n-k})$ is an integer.
		Without loss of generality, we can assume that $k\leq \lfloor n/2 \rfloor$. Then we have:
		\begin{align*}
			\frac{B_n}{B_kB_{n-k}} &= \frac{\prod_{i=1}^{n}y(i)^{x(n,i)}}{\prod_{i=1}^{k}y(i)^{x(k,i)}\prod_{i=1}^{n-k}y(i)^{x(n-k,i)}}\\
			&=  \prod_{i=1}^{k}y(i)^{x(n,i)-x(k,i)-x(n-k,i)} \prod_{i=k+1}^{n-k}y(i)^{x(n,i)-x(n-k,i)} \prod_{i=n-k+1}^{n}y(i)^{x(n,i)}.
		\end{align*}
		By the construction of $x(n,k)$, the power of each $y_i$ in the above expression is a non-negative integer.
	\end{proof}
	\begin{theorem}
		Let $x=(x(n))_{n\geq 1}$ be an unbounded sequence of positive integers, and $y=(y(n))_{n\geq 1}$ be a sequence of positive integers. The sequence $(B_n)_{n\geq 0}$ defined in Definition \ref{defi of factorial set Bn} has the following form:
		
		\begin{equation}
			 B_n=\prod_{k=1}^{n}y(k)^{\hat{x}(\lfloor n/k \rfloor)}.
		\end{equation}
		Where $\lfloor \cdot \rfloor$ denotes the floor function, and $(\hat{x}(n))_{n\geq 1}$ is the minimal superadditive subsequence of $x$ as defined in Definition \ref{defi of minsuper sequence}.
		In other words, the power of $y_k$ in the expression for $B_n$ is given by $x(n,k)=\hat{x}(\lfloor n/k \rfloor)$.
	\end{theorem}
	\begin{proof}
		For fixed $k$, we prove by induction on $n$ that $\hat{x}(\lfloor n/k \rfloor)=x(n,k)$.
		
		When $n=1$, $\hat{x}(\lfloor 1/k \rfloor)=x(1,k)$, since both of them are $min(\text{Im}(x))$ or zero depending on $k=1$ or $k>1$, respectively.
		
		Assume for all $k$ and $1\leq m \leq n-1$, $\hat{x}(\lfloor m/k \rfloor)=x(m,k)$. Note $(\hat{x}(n))_{n\geq 1}$ is increasing, so $\hat{x}(n_1)\leq \hat{x}(n_2)$ for $n_1\leq n_2$. By superadditive property of $\hat{x}$, we have 
		\[\hat{x}(\lfloor i/k \rfloor) + \hat{x}(\lfloor (n-i)/k \rfloor) \leq \hat{x}(\lfloor i/k \rfloor + \lfloor (n-i)/k \rfloor) \leq \hat{x}(\lfloor n/k \rfloor), 1\leq i \leq n-1.  
		\] 
		Here we have utilized the fact that $\lfloor x \rfloor + \lfloor y \rfloor \leq \lfloor x+y \rfloor$.
		By induction hypothesis, $\hat{x}(\lfloor i/k \rfloor)=x(i,k)$ and $\hat{x}(\lfloor (n-i)/k \rfloor)=x(n-i,k)$ for all $1\leq i \leq n-1$. So we have
		\[x(i,k)+x(n-i,k) \leq \hat{x}(\lfloor n/k \rfloor), 1\leq i \leq n-1.
		\]
		By definition, $x(n,k)$ is the smallest $a \in Im(x) $ such that $x(i,k)+x(n-i,k) \leq a$, we conclude $x(n,k)\leq \hat{x}(\lfloor n/k \rfloor)$.
		
		To prove the inverse inequality, by definition \ref{defi of minsuper sequence}, 
		\[\hat{x}(\lfloor n/k \rfloor)=\min_{a\in Im(x)}\{\hat{x}(m_1)+\hat{x}(\lfloor n/k \rfloor -m_1) \leq a, 1\leq m_1 \leq \lfloor n/k \rfloor -1\}.
		\]
		For every $m_1$ such that $1\leq m_1 \leq \lfloor n/k \rfloor -1$, we choose an integer $i$ such  that $m_1=\lfloor i/k \rfloor $, specifically, let $i=m_1k$, by induction hypothesis:
		\[\hat{x}(m_1)+\hat{x}(\lfloor n/k \rfloor -m_1)=\hat{x}(\lfloor i/k\rfloor)+\hat{x}(\lfloor (n-i)/k \rfloor)=x(i,k)+x(n-i,k)\leq x(n,k).
		\]  
		Since the above inequality holds for every $m_1$, we conclude $\hat{x}(\lfloor n/k \rfloor)\leq x(n,k)$. 
		
		Therefore, $\hat{x}(\lfloor n/k \rfloor) = x(n,k)$, completing the induction.
	\end{proof}
	So far based on two sequences $x,y$, we have defined two kinds of generalized factorials, as stated in Theorem \ref{the general definition of abstract factorial} and Theorem \ref{the simpler definition of abstract factorial}, respectively. In the remainder of this paper, whenever we refer to factorials associated with $(x,y)$, we specifically mean the construction in theorem \ref{the simpler definition of abstract factorial}. 
	
\subsection{Examples and applications}
	The following two examples can be found in \cite{mingarelli2013abstract}, where A.B.Mingarelli has provided a more detailed verification.
	\begin{example}
		Let $x(n)=n$ and $y(n)=e^{\Lambda(n)}$, the exponential of Von Mangoldt function, in this case the $n$-th factorial reduces to the ordinary factorial $n!$, 
		\begin{equation}\label{ordinary factorial representation}
			n!=\prod_{k=1}^{n} e^{\Lambda(k)\lfloor n/k \rfloor}.
		\end{equation}
		We remark that above formula is alternative form of well known de Polignac's formula:
		\[\text{Ord}_{p}(n!)=\sum_{i=1}^{\infty}\left\lfloor \frac{n}{p^i} \right\rfloor,
		\]
		where $p $ is prime and $\text{Ord}_{p}(m)$ denotes the highest power $d$ such that $p^d$ divides $m$.
	\end{example}
	\begin{example}
		Let $x(n)=n$, $\mathbb{P}$ be the set of all prime numbers, and $y(n)$ defined as follows:
		\[y(n)=\begin{cases}1 & \text{ if } n\neq p^m(p-1) \text{for any prime } p \text{ and } m\geq 0.\\
			\prod_{p\in \mathbb{P }, n=p^m(p-1)}p & \text{ if } n=p^m(p-1) \text{ for some prime } p \text{ and } m\geq 0.
		\end{cases}
		\]
		Then $(x,y)$ generate Bhargava's factorial \cite{Bhargava2000factorial}:
		\[n!_{x,y}=\prod_{p\in \mathbb{P}} p^{\sum_{m=0}^{\infty}\left \lfloor \frac{n}{p^m(p-1)}\right \rfloor}.
		\]
		It should be noted that $n!_{x,y}$ written above is actually the factorial of $n+1$ according to Bhargava's construction.
	\end{example}
	\begin{example}
		Let $x(n)=n^2$ and $y(n)=e^{\Lambda(n)}$, we have
		\[n!_{x,y}=\prod_{k=1}^{n}e^{\Lambda(k)\left\lfloor n/k\right\rfloor^2}=\prod_{p\in \mathbb{P}}p^{\sum_{i\geq 1}\left\lfloor \frac{n}{p^i}\right\rfloor^2}
		\]
		This factorial has a simple and straightforward interpretation:
		\begin{equation}\label{x(n)=n^2 factorial}\prod_{p\in \mathbb{P}}p^{\sum_{i\geq 1}\left\lfloor \frac{n}{p^i}\right\rfloor^2}=
		\prod_{j=1}^{n}\prod_{k=1}^{n} \text{g.c.d}(j,k).
		\end{equation}
		Where $\text{g.c.d}(j,k)$ is the greatest common divisor of $j,k$. 
		
		 Formula \eqref{x(n)=n^2 factorial} was conjectured in \cite{oeis} (A092287), and confirmed by 
		 C.R. Greenhouse within the same sequence. Here we provide a short proof of this fact.
		Consider a prime $p$, $j\leq n$ and $k\leq n$, we can determine when $p$ is a factor of $\text{g.c.d}(j,k)$. Specifically, $p$ divides $\text{g.c.d}(j,k)$ if and only if $j$ belongs to $\{p,2p,3p,...,\left\lfloor n/p\right\rfloor p\} $ and $k\in \{p,2p,3p,...,\left\lfloor n/p\right\rfloor p\}$, the total number of such pairs is $\left\lfloor n/p\right\rfloor^2$. Next consider when $p^2$ divides $\text{g.c.d}(j,k)$. In this case we require $j\in \{p^2,2p^2,3p^2,...,\left\lfloor n/p^2\right\rfloor p^2\}$ and $k\in \{p^2,2p^2,3p^2,...,\left\lfloor n/p^2\right\rfloor p^2\}$, resulting in a total of $\left\lfloor n/p^2\right\rfloor^2$ pairs. Continuing this reasoning, we find that the total number of $p$ divides the product on the right of \eqref{x(n)=n^2 factorial}is 
		\[\text{Ord}_p\left(\prod_{j=1}^{n}\prod_{k=1}^{n} \text{g.c.d}(j,k)\right)=\sum_{i=1}^{\infty}\left\lfloor \frac{n}{p^i}\right\rfloor^2.
		\]  
		which aligns with our interpretation.
		
		It is straightforward to extend \eqref{x(n)=n^2 factorial} to a general case, for each positive integer $q\geq 1$, let $x(n)=n^q$, $y(n)=e^{\Lambda(n)}$, the factorial $n!_{x,y}$ has the following interpretation:
		\[n!_{x,y}=\prod_{p\in \mathbb{P}}p^{\sum_{i\geq 1}\left\lfloor \frac{n}{p^i}\right\rfloor^q}=
		\prod_{1\leq j_1\leq n}\prod_{1\leq j_2\leq n}\ldots\prod_{1\leq j_q \leq n} \text{g.c.d}(j_1,j_2,\ldots,j_q).
		\]
	\end{example}
	\begin{proof}
		The proof of this fact is essentially same with the $q=2$ case, mentioned above.
	\end{proof}
	\begin{definition}
		For a generalized factorial $(n!_E)_{n\geq 0}$, the associated Euler constant is defined by \cite{mingarelli2013abstract}: 
		\[e_E:=\sum_{n=0}^{\infty}\frac{1}{n!_E}.
		\]
	\end{definition}
	Mingarelli \cite{mingarelli2013abstract} has proved the following:
	\begin{theorem}
		Let $(n!_E)_{n\geq 0}$ is a generalized factorial, and $e_E$ is the Euler constant associated to $(n!_E)_{n\geq 0}$,
		then $e_E$ is irrational. In fact if $!_{E_1},!_{E_2},\cdots,!_{E_k}$ is any collection of factorial functions, for $s_1,s_2,\cdots,s_k\in \mathbb{N}$, not all zero, then
		\[\sum_{n=0}^{\infty}\frac{1}{\prod_{j=1}^{k}n!_{E_j}^{s_j}}, \quad \sum_{n=0}^{\infty}\frac{(-1)^n}{\prod_{j=1}^{k}n!_{E_j}^{s_j}}
		\] 
		are each irrational. 
	\end{theorem}
	Apply above theorem to our construction of generalized factorial, we have 
	\begin{corollary}
		For each $x_i=(x_i(n))_{n\geq 1}, 1\leq i \leq k$ is a superadditive sequence of positive integers, and $y(n)=e^{\Lambda(n)}$, then  for $s_1,s_2,\cdots,s_k\in \mathbb{N}$, not all zero,
		\[\sum_{n=0}^{\infty}\frac{1}{\prod_{j=1}^{k}n!_{x_j,y}^{s_j}}, \quad \sum_{n=0}^{\infty}\frac{(-1)^n}{\prod_{j=1}^{k}n!_{x_j,y}^{s_j}}
		\]
		are each irrational.
	\end{corollary}
	\begin{proof}
		one just need to verify each $!_{x_i,y},1\leq i \leq k$ is indeed factorial, this is due to $x_i(n)\geq n$, since we assume each $x_i$ is superadditive.
	\end{proof}
	
\section{The inverse problem}
	Since we have shown that for any superadditive arithmetic function $(x(n))_{n\geq 1}$ and
	$(y(n))_{n\geq 1}$ of positive integers, there is a generalized factorial $(n!_{x,y})_{n\geq 0}$ associated with them. One may inquire about the inverse problem, given a generalized factorial $(n!_E)_{n\geq 0}$, does there exist a pair $(x,y)$ that could recover it in Theorem \ref{the simpler definition of abstract factorial}? Furthermore, if such a pair exists, is it uniquely determined ?
	
	For an arithmetic function (or sequence of positive integers) $\alpha$, denote difference function $\Delta\alpha$ as follows: 
	\[\Delta\alpha(n):=\alpha(n)-\alpha(n-1),\; n\geq 1.
	\]
	Where we set $\alpha(0)=0$.
	
	For arithmetic function $\beta$, denote logarithm function $\text{Log}(\beta)$ by 
	\[\text{Log}(\beta)(n):=\log(\beta(n)).
	\] 
	We present the following theorem that characterize the uniqueness of the factorial:
	\begin{theorem}\label{characterization theorem}
		Two pair of arithmetic functions $(\alpha,\beta)$ and $(x,y)$ generate same set of factorials in Theorem \ref{the simpler definition of abstract factorial} if and only if the Dirichlet convolution satisfies
		\begin{equation}\label{convolution characterization}
			\Delta\alpha \star Log(\beta)=\Delta x \star Log(y). 
		\end{equation}
	\end{theorem}
	\begin{proof}
		First suppose $\Delta \alpha\star Log(\beta)=\Delta x \star Log(y)$. By definition of convolution, for every $k\geq 1$:
		\begin{equation}
			\sum_{d | k} log(\beta(d))\left(\alpha\left(\frac{k}{d}\right)-\alpha\left(\frac{k}{d}-1\right)\right)=\sum_{d | k} log(y(d))\left(x\left(\frac{k}{d}\right)-x\left(\frac{k}{d}-1\right)\right).
		\end{equation}
		Take exponential of both sides, we obtain
		\begin{equation}\label{exponenial of convolution}
			\prod_{d|k}\beta(d)^{\alpha(\frac{k}{d})-\alpha(\frac{k}{d}-1)}=\prod_{d|k}y(d)^{x(\frac{k}{d})-x(\frac{k}{d}-1)}, \; k\geq 1.
		\end{equation}
		multiply the above equation for $k=1,2,\ldots,n$, we have
		\begin{equation}\label{eqn An=Bn form}
			\prod_{k=1}^{n}\prod_{d|k}\beta(d)^{\alpha(\frac{k}{d})-\alpha(\frac{k}{d}-1)}=
			\prod_{k=1}^{n}\prod_{d|k}y(d)^{x(\frac{k}{d})-x(\frac{k}{d}-1)},\; n\geq 1.
		\end{equation}
		On the left side, interchange the order of product to obtain 
		\begin{equation}\prod_{k=1}^{n}\prod_{d|k}\beta(d)^{\alpha(\frac{k}{d})-\alpha(\frac{k}{d}-1)}
			=\prod_{d=1}^{n}\prod_{\substack{k\leq n \\ d|k}}\beta(d)^{\alpha(\frac{k}{d})-\alpha(\frac{k}{d}-1)}, \; n\geq 1.
		\end{equation}
		For $k\leq n$, $d$ divides $k$ if and only if $k\in \{d,2d,\ldots, \left\lfloor n/d\right\rfloor d\}$, hence
		\begin{align}
		 \prod_{d=1}^{n}\prod_{\substack{k\leq n \\ d|k}}\beta(d)^{\alpha(\frac{k}{d})-\alpha(\frac{k}{d}-1)} 
		 &= \prod_{d=1}^{n} \beta(d)^{\alpha(\frac{d}{d})-\alpha(\frac{d}{d}-1)+\ldots+\alpha\left(\frac{\left\lfloor n/d \right\rfloor d}{d}\right)-\alpha\left(\frac{\left\lfloor n/d \right\rfloor d}{d}-1\right)}\\
		     &= \prod_{d=1}^{n}\beta(d)^{\alpha(\left\lfloor n/d \right\rfloor)}\\
		     &= n!_{\alpha,\beta}.\label{final form of !(a,b)=!(x,y)}
		\end{align}
		Therefore by \eqref{eqn An=Bn form} $(\alpha,\beta)$ and $(x,y)$ generate same set of factorials. 
		
		Next, suppose $(\alpha,\beta)$ and $(x,y)$ generate same set of factorials $n!_E$, by definition 
		\[\prod_{k=1}^{n}\beta(k)^{\alpha\left(\left\lfloor \frac{n}{k} \right \rfloor\right)}=n!_E=\prod_{k=1}^{n}y(k)^{x\left(\left\lfloor \frac{n}{k} \right \rfloor\right)}
		\]
		We can reverse the argument from \eqref{final form of !(a,b)=!(x,y)} to get \eqref{eqn An=Bn form}, note \eqref{eqn An=Bn form} has the form $A_n=B_n$, from 
		\begin{equation}
			\frac{A_n}{A_{n-1}}=\frac{n!_E}{(n-1)!_E}=\frac{B_n}{B_{n-1}}.
		\end{equation}
		We get \eqref{exponenial of convolution}, which is equivalent to \eqref{convolution characterization}.
	\end{proof}
	\begin{corollary}\label{corollary alpha=id}
		For Two pair of arithmetic functions $(Id,\beta)$ and $(x,y)$, they generate same set of factorials if and only if the Dirichlet convolution satisfies
		\begin{equation}\label{special convolution characterization}
			1 \star Log(\beta)=\Delta x \star Log(y). 
		\end{equation}
	\end{corollary}
	 
	\begin{proof}
		Since $Id(n)=n$, $\Delta Id(n)=Id(n)-Id(n-1)=1=1(n)$.
		\end{proof}
	\begin{corollary}
		If $n!_E$ is generated by $(\alpha,\beta)$, then
		\begin{equation}
			\frac{n!_E}{(n-1)!_E}=\prod_{d|n}\beta(d)^{\alpha(\frac{n}{d})-\alpha(\frac{n}{d}-1)},n\geq1.
		\end{equation}
	\end{corollary}
	\begin{proof}
		This is simply a restatement of equations \eqref{exponenial of convolution} and \eqref{eqn An=Bn form}.
	\end{proof}
	Whenever $(\alpha,\beta)$ and $(x,y)$ generate the same factorial sequence, and we know three of these functions, we can use theorem (\ref{characterization theorem}) to determine the fourth one. If a factorial $!_E$ could be represented by $(\alpha=Id,\beta)$ in theorem \ref{the simpler definition of abstract factorial}, we refer to $(Id,\beta)$ as a "standard" representation of $!_E$.
	\begin{example}
		Let $x(n)=n(n+1)/2$, $y(n)=e^{\Lambda(n)}$, suppose we want to find $\beta$ such that $n!_{x,y}$ has the "standard" representation:
		\[n!_{x,y}=\prod_{k=1}^{n} \beta(k)^{\left\lfloor n/k \right\rfloor}.
		\]
		By corollary \ref{corollary alpha=id}, we have 
		\[1\star Log(\beta) = Id \star \Lambda
		\]
		Since $\Delta x=Id$ appears on the right side of above equation. Now the Dirichlet inverse of $1$ is the M\"{o}bius function $\mu$, which allow us to solve for $\beta$ :
		\[\beta(n)=\prod_{d|n} e^{\Lambda(d)(\mu\star Id)(\frac{n}{d})} ,n\geq 1.
		\]  
		Notice $\mu\star Id$ is simply Euler's totient function $\phi $, we conclude that $\beta(n)$ is always a positive integer, hence the factorial $n!_{x,y}$ also generated by $\alpha=Id$ and $\beta=e^{\Lambda\star \phi}$.
	\end{example}
	For each positive integer $k$, Jordan's totient function \cite{sandor2006handbook} $J_k$ is defined as 
	\[J_k(n)=n^k\prod_{p|n}\left(1-\frac{1}{p^k}\right),
	\] 
	where $p$ ranges over all the prime divisors of $n$.
	It can be verified that $J_1$ reduce to the Euler's totient function, further we have $\mu(n) \star n^k =J_k(n)$.
	\begin{example}
		Let $y(n)=e^{\Lambda(n)}$, and $x(n)=\sum_{i=1}^{n}i^k$, the sum of $k$'s power of first $n$ positive integers.
		If we want to express $n!_{x,y}$ in standard form, then again by \ref{corollary alpha=id}, function $\beta$,  is given by 
		\[\beta=e^{\Lambda \star \Delta x \star \mu}=e^{\Lambda \star J_k}.
		\]  
	\end{example}
	Next we show that not every generalized factorial allows a standard representation. by Corollary \ref{corollary alpha=id}, we compute the first few elements of $\beta$:
	\begin{align}
		\beta(1) & = y(1)^{x(1)};\label{beta 1 value}\\ 
		\beta(2) & = y(1)^{-2x(1)+x(2)}y(2)^{x(1)};\label{beta 2 value}\\
		\beta(3) & = y(1)^{-x(1)-x(2)+x(3)} y(3)^{x(1)};\label{beta 3 value}\\
		\beta(4) & = y(1)^{x(1)-x(2)-x(3)+x(4)}y(2)^{-2x(1)+x(2)}y(4)^{x(1)}\label{beta 4 value}.
	\end{align}
	The issue arises when there exists some superadditive functions $x$ such that $x(1)-x(2)-x(3)+x(4)<0$. For example we can define $x(1)=1, x(2)=3, x(3)=5, x(4)=6$, while choosing $y(1)=4, y(2)=2, y(3)=12, y(4)=1$. One can verify function $x$ is superadditive for these first four values, and $n! | n!_{x,y}, 1\leq n\leq 4$. This would lead to $\beta(4)=1/2$, which is not an integer. 
	
	This paper has demonstrated the method of combining two arithmetic functions to produce a generalized factorial. We have presented several examples beyond the conventional factorial, including a noteworthy result on the irrationality of Euler's number. Furthermore, we have shown that the logarithm of these generalized factorials can be expressed as the Dirichlet convolution of two arithmetic functions.
	
	Our findings open up numerous avenues for future research. Some areas for exploration include investigating alternative representations of the usual factorial, defining Stirling numbers associated with these generalized factorials, and developing a generalized version of binomial inversion within this context.  
	

\subsection*{Acknowledgment}
I would like to thank the anonymous referee for a careful reading of the paper.

\end{document}